\documentclass[10pt]{amsart}  

\usepackage{hyperref}
\usepackage{amsmath,amssymb}
\usepackage{mathtools}
\usepackage{amscd} 
\usepackage{mathrsfs}
\usepackage{bbold}
\usepackage{amsthm}
\numberwithin{equation}{section}
\newtheorem{theorem}{Theorem}[section]
\newtheorem{lemma}[theorem]{Lemma}

\theoremstyle{definition}
\newtheorem{remark}[theorem]{Remark}

\newtheorem{example}[theorem]{Example}

\def\B{\mathscr B}
\def\Q{\mathscr Q}
\def\H{\mathsf H}
\def\K{\mathsf K}
\def\X{\mathsf X}
\def\Y{\mathsf Y}
\def\C{\mathscr C}
\def\dom{\text{\rm dom}}
\def\ran{\text{\rm ran}}
\def\RE{\mathbb R}
\def\CO{{\mathbb C}}
\def\uno{\mathbb 1}
\def\WW{W}
\def\VV{V}

\def\be{\begin{equation}}
\def\ee{\end{equation}}

\begin{document}
\renewcommand{\subjclassname}{%
\textup{2010} Mathematics Subject Classification}

\title[On inverses of Krein's $\Q$-functions]{On inverses of Krein's $\Q$-functions}


\author{Claudio Cacciapuoti}
\address{DiSAT, Sezione di Matematica, Universit\`a dell'Insubria, via Valleggio 11, I-22100
Como, Italy}
\email{claudio.cacciapuoti@unisubria.it}

\author{Davide Fermi}
\address{Dipartimento di Matematica, Universit\`a di Milano, Via Cesare Saldini 50, I-20133 Milano, Italy }
\email{davide.fermi@unimi.it}

\author{Andrea Posilicano}
\address{DiSAT, Sezione di Matematica, Universit\`a dell'Insubria, via Valleggio 11, I-22100
Como, Italy}
\email{andrea.posilicano@unisubria.it}

\dedicatory{Dedicated to Gianfausto Dell'Antonio on the occasion of his 85th birthday}

\begin{abstract} Let $A_{Q}$ be the self-adjoint operator defined by the $\Q$-function $Q:z\mapsto Q_{z}$  through the Kre\u\i n-like resolvent formula $$(-A_{Q}+z)^{-1}=
(-A_{0}+z)^{-1}+G_{z}WQ_{z}^{-1}VG_{\bar z}^{*}\,,\quad z\in Z_{Q}\,,$$ where  $\VV$ and $\WW$ are bounded operators and $$Z_{Q}:=\{z\in\rho(A_{0}):\text{$Q_{z}$ and $Q_{\bar z }$ have a bounded inverse}\}\,.$$ 
We show that $$Z_{Q}\not=\emptyset\quad\Longrightarrow\quad Z_{Q}=\rho(A_{0})\cap \rho(A_{Q})\,.$$ 
We do not suppose that $Q$ is represented in terms of a uniformly strict, operator-valued Nevanlinna function (equivalently, we do not assume that $Q$ is associated to an ordinary boundary triplet), thus our result extends previously known ones. The proof relies on simple algebraic computations stemming from the first resolvent identity.
\end{abstract}

\maketitle

\begin{section}{Introduction}
Let $A_{0}:\dom(A_{0})\subseteq\H\to\H$ be a self-adjoint operator in the Hilbert space $\H$ and let $S:\dom(S)\subseteq\H\to\H$ be the symmetric operator given by the restriction of $A_{0}$ to the kernel (assumed to be dense) of the continuous (w.r.t. the graph norm) linear map $\tau:\dom(A_{0})\to \K$, $\K$ being an auxiliary Hilbert space.  By \cite[Theorem 2.1]{P01} (see Theorem \ref{teo} in the next section), a family of self-adjoint extensions of $S$ can be defined through the Kre\u\i n-like resolvent formula 
\be\label{krein}
(-A_{Q}+z)^{-1}=
(-A_{0}+z)^{-1}+G_{z}\WW Q_{z}^{-1}\VV G_{\bar z}^{*}\,,\quad z\in Z_{Q}\,,
\ee
where $\VV$ and $\WW$ are bounded operators, 
$$Z_{Q}:=\{z\in\rho(A_{0}):\text{$Q_{z}$ and $Q_{\bar z }$ have a bounded inverse}\}$$ and $Q_{z}$ is a family of (not necessarily bounded) densely defined, closed linear maps such that 
\[
Q_{z}-Q_{w}=(z-w)\VV\tau(-A_{0}+w)^{-1}(\tau(-A_{0}+\bar z)^{-1})^{*}\WW\,,\qquad w,z\in\rho(A_{0})\,,
\]
and 
\be\label{Q2}
\VV^{*}(Q_{z}^{*})^{-1}\WW^{*}=\WW Q^{-1}_{\bar z}\VV\,,\qquad z\in Z_{Q}\,.
\ee 
By a slight abuse of terminology, we call such a map $Q:z\mapsto Q_{z}$ a $\Q$-function; for  the definition (in the case $V=\uno$ and $W=\uno$, where \eqref{Q2} reduces to $Q^{*}_{z}=Q_{\bar z}$) of a ``$\Q$-function of $S$ belonging to $A_{0}$'' (with values in the space of bounded operators) we refer to \cite[Definition 3]{DM} and to the original papers \cite{K44} (defect indices $n_{\pm}(S)=1$), \cite{K46} (finite defect indices), \cite {Saa} (infinite defect indices). Evidently the above definition of $A_{Q}$ by \eqref{krein} only requires $Z_{Q}\not=\emptyset$.  However, taking into account formula \eqref{krein}, one would expect $\rho(A_{0})\cap\rho(A_{Q})\subseteq Z_{Q}$ (hence $Z_{Q}=\rho(A_{0})\cap\rho(A_{Q})$ since $Z_{Q}\subseteq \rho(A_{0})\cap\rho(A_{Q})$ by $Z_{Q}\subseteq \rho(A_{0})$ and \eqref{krein}); moreover, in order to treat scattering theory for the couple $(A_{Q},A_{0})$ through a limiting absorption principle (see \cite{MPS}, \cite{MPS-JDE2}, \cite{MP}, \cite{CFP}), one at least would need $\CO\backslash\RE\subseteq Z_{Q}$. The aim of this work is to show that if $Z_{Q}$ is not empty then it  necessarily coincides with $\rho(A_{0})\cap\rho(A_{Q})$ (and so it always contains the whole $\CO\backslash\RE$).  
In the case the map $\tau$ is surjective, i.e., $\ran(\tau)=\K$, and $\VV=\pi$, $\WW=\pi^{*}$, $\pi$ an orthogonal projector onto a closed subspace of $\K$ (coinciding with $\K$ itself in the case $\pi=\uno$), then (see \cite{P04}, \cite[Section 4]{P08}) the construction provided in \cite{P01} is equivalent to the one given by boundary triplet theory (we refer to \cite{DM}, \cite[Section 1]{BGP}, \cite[Section 7.3]{DHMdS}, \cite[Section 14]{Schm} and references therein for an introduction to such a theory). Thus, in this case, $Q$ can be expressed in terms of a self-adjoint operator  and an holomorphic function $M:z\mapsto M_{z}$ with values in the space of bounded operators
 such that $M_{z}=M^{*}_{\bar z}$ and $0\in\rho(M_{z}-M_{z}^{*})$ (see \cite[Theorem 1]{DM}, \cite [Theorem 7.15]{DHMdS}), i.e., $M$ is a uniformly strict Nevanlinna operator function. Hence, whenever $\ran(\tau)=\K$, $\VV=\pi$,  $\WW=\pi^{*}$, one gets $Z_{Q}=\rho(A_{0})\cap\rho(A_{Q})$ by standard arguments  (see \cite[Theorem 2]{DM}, \cite [Theorem 7.16]{DHMdS}; see also \cite[Proposition 2.1]{P01}, \cite[Theorem 2.1]{P08}). Since, by the correspondence with von Neumann's theory (see \cite{P03}, \cite{P08}), any self-adjoint extension of $S$ can be defined through \eqref{krein} assuming the hypothesis $\ran(\tau)=\K$ (equivalently, using the corresponding ordinary boundary triplet, see \cite{DM}, \cite[Theorem 14.7]{Schm}), these results seems to settle down our questions about $Z_{Q}$ (at least in the case $\VV=\pi$,  $\WW=\pi^{*}$).  
However, in cases where the defect indices of $S$ are not finite, in particular in applications to partial differential operators,  it can be much more convenient to do not require $\ran(\tau)=\K$ (and sometimes $\VV\not=\uno$, $\WW\not=\uno$) and so to do not use ordinary boundary triplets (see, e.g., \cite{DFT}, \cite{BEK}, \cite{DFT97}, \cite{P01}, \cite{Kond}, \cite{Ex}, \cite{EK}, \cite{EK05}, \cite{BL07}, \cite{BL}, \cite{BLL}, \cite{BM}, \cite{MP}, \cite{CFP-INdAM}, \cite{CFP}). While some results regarding the validity of \eqref{krein} for any $z\in\rho(A_{0})\cap\rho(A_{Q})$ are known even for not ordinary boundary triplets (as generalized boundary triplets and quasi-boundary triplets, see e.g., \cite{BL}, \cite{DHMdS}, \cite{BLLR}), some additional hypotheses are required in these cases (which moreover do not necessarily conform to our framework). Here, see Theorem \ref{TeoInv} in the next section, we provide a simple proof of $$Z_{Q}\not=\emptyset\quad\Rightarrow\quad Z_{Q}=\rho(A_{0})\cap\rho(A_{Q})$$ in the case $\ran(\tau)\not=\K$, $\VV\not=\pi$, $\WW\not=\pi^{*}$, without further hypotheses. 

\end{section}
\begin{section}{Inverses of Krein's $\Q$-functions}

Let $\H$ and ${\K}$ be Hilbert spaces with scalar products (which we assume to be conjugate-linear w.r.t. the first variable)
$\langle\cdot,\cdot\rangle_{\H}$ and $\langle\cdot,\cdot\rangle_{{\K}}$. In the following, for notational convenience, we do not identify ${\K}$ with its dual ${\K}^{*}$; however we use ${\K}^{**}\equiv{\K}$. We denote by $\langle\cdot,\cdot\rangle_{{\K}^{*}\!,{\K}}$ the ${\K}^{*}$-${\K}$ duality (conjugate-linear w.r.t. the first variable) defined by 
$\langle\psi,\phi\rangle_{{\K}^{*}\!,{\K}}:=\langle J^{-1}\psi,\phi\rangle_{{\K}}$, where $J:{\K}\to{\K}^{*}$ is the duality mapping given by the differential of  $\phi\mapsto\frac12\langle\phi,\phi\rangle_{{\K}}$.\par 
Given the self-adjoint operator $$A_{0}:\dom(A_{0})\subseteq\H\to\H\,,$$ we consider a continuous (w.r.t. the graph norm in $\dom(A_0)$) linear map $$\tau:\dom(A_0)\to{\K}$$ such that 
\be\label{dense}
\text{ker$(\tau)$ is dense in $\H$.}
\ee   
\begin{remark} Notice that we do not suppose that $\ran(\tau)=\K$. This means that the corresponding (accordingly to \cite{P04}) boundary triplet is not an ordinary boundary triplet. See the successive Remark \ref{nodense} for the case in which $\ker(\tau)$ is not dense.
\end{remark}
For any $z\in\rho(A_{0})$ we define $R^0_{z}\in\B(\H,\dom(A_0))$  by $R^{0}_{z}:=(-A_{0}+z)^{-1}$ and $G_{z}\in\B({{\K}}^{*},\H)$ by  
$$ G_{z}:{{\K}}^{*}\to\H\,, \quad G_{z}:=(\tau R^{0}_{\bar z})^*\,,
$$
i.e.,
\[
\langle G_{z}\phi,u\rangle_{\H}=\langle\phi,\tau(-A_{0}+\bar z)^{-1}u\rangle_{{\K}^{*}\!,{\K}}
\quad \phi\in{\K}^{*}\,,\ u\in\H\,.
\]
By \eqref{dense}, one has (see \cite[Remark 2.9]{P01}),
\[
\ran(G_{z})\cap\dom(A_0)=\{0\}
\]
and,  by the resolvent identity, 
\be\label{Gzw}
G_{z}-G_{w}=(w-z)R^0_{w}G_{z}\,,
\ee  
so that
\be\label{reg}
\ran(G_{z}-G_{w})\subset\dom(A_0)\,.
\ee
\begin{remark} Notice that \eqref{Gzw} is equivalent to 
\be\label{Gzw'}
G_{z}=\left(1+(w-z)R^0_{z}\right)G_{w}\,.
\ee 
\end{remark}
Let $\X$ and $\Y$ be two Hilbert spaces and let $\VV$ and $\WW$ be two bounded operators,   $\VV\in \B(\K,\X)$ and $\WW\in \B(\Y,{\K}^*)$. Given a not empty set $ Z_{\Lambda}\subseteq\rho(A_{0})$, symmetric with respect to the real axis (i.e., $z\in Z_{\Lambda}\Rightarrow \bar z\in Z_{\Lambda}$), we consider a map $$\Lambda:Z_{\Lambda}\to \B({\X} ,\Y)\,,\qquad z\mapsto \Lambda_{z}\,,
$$ 
such that
\be\label{Lambda1}
\VV^* \Lambda_{z}^* \WW^* =\WW \Lambda_{\bar z}\VV \,,
\ee
\be\label{Lambda2} 
\Lambda_{z}-\Lambda_{w}=(w-z)\Lambda_{w}\VV G_{\bar w}^{*}G_{z}\WW \Lambda_{z}\,.
\ee 
\begin{remark} Notice that \eqref{Lambda2} is equivalent to 
\be\label{Lambda2'} 
\Lambda_{z}=\left(1+(w-z)\Lambda_{z}\VV G_{\bar z}^{*}G_{w}\WW \right)\Lambda_{w}\,.
\ee 
\end{remark}
Notice that, by \eqref{Lambda1} and \eqref{Lambda2}, the map $\widetilde \Lambda_{z}:= \WW \Lambda_{ z}\VV : \K \to {\K}^*$ satisfies the relations 
 $$\widetilde \Lambda_{z}^*  =\widetilde \Lambda_{\bar z}$$ and $$\widetilde \Lambda_{z}-\widetilde \Lambda_{w}=(w-z)\widetilde \Lambda_{w}G_{\bar w}^{*}G_{z}\widetilde \Lambda_{z}\,,$$ see \cite[equations (2) and (4)]{P01}. Hence, 
building on \cite[Theorem 2.1]{P01},  one has (see \cite[Theorem 2.4 and Remark 2.5]{MP}; our $\widetilde \Lambda_{z}= \WW \Lambda_{ z}\VV$ corresponds to the operator there denoted by $\Lambda_{z}$)
\begin{theorem}\label{teo}
Let $\Lambda:Z_{\Lambda}\to \B({\X} ,{\Y})$ satisfy \eqref{Lambda1} and \eqref{Lambda2}. Then there exists a unique self-adjoint extension $A_{\Lambda}$ of the closed symmetric operator $S:=A_{0}|\ker(\tau)$ such that $Z_{\Lambda}\subseteq \rho(A_{0})\cap \rho(A_{\Lambda})$ and  
\be\label{resolvent}
(-A_{\Lambda}+z)^{-1}=R^0_{z}+G_{z}\WW \Lambda_{z}\VV G^{*}_{\bar z}\,,\qquad z\in Z_{\Lambda}\,.
\ee
\end{theorem}
\begin{remark} Any self-adjoint extension of $S$ is of the kind provided by the previous theorem (see \cite{P03}, \cite{P08}).
\end{remark}
From now on we use the shorthand notation 
$$
R_{z}^{\Lambda}:=(-A_{\Lambda}+z)^{-1}\,,\qquad z\in \rho(A_{\Lambda})\,.
$$
\begin{lemma}  For any $w$ and $z$ in $Z_{\Lambda}$ one has
\be\label{Lambda3}
\Lambda_{z}-\Lambda_{w}=(w-z)\Lambda_{w}\VV G_{\bar w}^{*}\left(1+(w-z)R^{\Lambda}_{z}\right)G_{w}\WW \Lambda_{w}\,.
\ee
\end{lemma}
\begin{proof} Taking into account relations \eqref{Lambda2}, \eqref{Lambda2'}, \eqref{Gzw'}  and \eqref{resolvent}, one gets 
\begin{align*}
&\Lambda_{z}-\Lambda_{w}\\
=&(w-z)\Lambda_{w}\VV G_{\bar w}^{*}\left(G_{z}+
(w-z)G_{z}\WW \Lambda_{z}\VV G_{\bar z}^{*}G_{w}\right)\WW \Lambda_{w}\\
=&
(w-z)\Lambda_{w}\VV G_{\bar w}^{*}\left(1+(w-z)R^{0}_{z}+(w-z)G_{z}\WW \Lambda_{z}\VV G_{\bar z}^{*}\right)G_{w}\WW \Lambda_{w}\\
=&(w-z)\Lambda_{w}\VV G_{\bar w}^{*}\left(1+(w-z)R^{\Lambda}_{z}\right)G_{w}\WW \Lambda_{w}\,.\end{align*}
\end{proof}
Obviously, by \eqref{resolvent}, 
$\rho(A_{\Lambda})\ni z\mapsto R^{\Lambda}_{z}$ is a $\B(\H)$-valued analytic extension of $Z_{\Lambda}\ni z\mapsto R^0_{z}+G_{z}\WW \Lambda_{z}\VV G^{*}_{\bar z}$. Thus, given $w\in Z_{\Lambda}$, relation \eqref{Lambda3} suggests to define an analytic extension  of $\Lambda$ by
$$\widehat\Lambda^{(w)}:\rho(A_{\Lambda})\to\B({\X},{\Y})\,,$$
\be\label{Lambda4}
\widehat\Lambda^{(w)}_{z}:=\Lambda_{w}+(w-z)\Lambda_{w}\VV G_{\bar w}^{*}\left(1+(w-z)R^{\Lambda}_{z}\right)G_{w}\WW \Lambda_{w}\,.
\ee
\begin{lemma}  Suppose that $Z_{\Lambda}$ contains at least an accumulation point. Then $\widehat\Lambda^{(w)}$ is $w$-independent.
\end{lemma}
\begin{proof} Let $w_{1}\not=w_{2}$. At first suppose that $A_{\Lambda}$ has a spectral gap, equivalently $\rho(A_{\Lambda})$ is a connected subset of $\CO$. Since 
$\widehat\Lambda^{(w_{1})}=\widehat\Lambda^{(w_{2})}$ on $Z_{\Lambda}$ by \eqref{Lambda3}, then 
$\widehat\Lambda^{(w_{1})}=\widehat\Lambda^{(w_{2})}$ on the whole $\rho(A_{\Lambda})$ by the Identity Theorem for analytic functions. Conversely suppose that $\rho(A_{\Lambda})=\CO_{-}\cup\CO_{+}$, where $\CO_{\pm}:=\{z\in\CO:\pm\text{Im}(z)>0\}$. Then the thesis is consequence of the same argument separately applied to the connected sets $\CO_{-}$ and $\CO_{+}$.   
\end{proof}
\begin{remark}\label{Lambda*} Suppose that $\widehat\Lambda^{(w)}$ in \eqref{Lambda4} does not depend on the choice of $w\in Z_{\Lambda}$, $\widehat\Lambda_{z}\equiv\widehat\Lambda_{z}^{(w)}$; then $\VV^{*}\widehat\Lambda_{z}^{*}\WW^{*}=\WW\widehat\Lambda_{\bar z}\VV$ : by \eqref{Lambda1} and $(R^{\Lambda}_{z})^{*}=R^{\Lambda}_{\bar z}$, one has
$$
\VV^{*}\widehat\Lambda_{z}^{*}\WW^{*}=\WW\Lambda_{\bar w}\VV+(\bar w-\bar z)\WW\Lambda_{\bar w}\VV G_{ w}^{*}\left(1+(\bar w-\bar z)R^{\Lambda}_{\bar z}\right)G_{\bar w}\WW\Lambda_{\bar w}\VV=\WW\widehat\Lambda_{\bar z}\VV\,.
$$
\end{remark}
The previous lemma suggests that the Kre\u\i n-like resolvent formula \eqref{resolvent} could hold on a larger set, i.e., 
$$
(-A_{\Lambda}+z)^{-1}=R^0_{z}+G_{z}\WW \widehat \Lambda_{z}\VV  G^{*}_{\bar z}\,,\qquad z\in \rho(A_{0})\cap \rho(A_{\Lambda})\,.
$$
Let us consider a map $$Q:\rho(A_{0})\to \C({\Y} ,{\X})\,,\qquad z\mapsto Q_{z}\,,
$$ (here $\C({\Y} ,{\X})$ denotes the set of closed linear operators) such that 
\be\label{M0}
\text{$\dom(Q_{z})$ is $z$-independent, $\dom(Q_{z})\equiv{\mathsf D}$, and dense, $\overline{{\mathsf D}}=\Y $,}\ee
\be\label{M2}
Q_{z}=Q_{w}+(z-w)\VV G_{\bar w}^{*}G_{z}\WW\qquad  z,w \in\rho (A_0)\,.
\ee
Defining
\[\begin{aligned}
Z_{Q}&&\\
:=&\{z\in\rho(A_{0}):\text{$Q_{z}$ and $Q_{\bar z}$ are bijections from ${\mathsf D}$ onto $\X$ with inverses in $\B({\X},{\Y})$}\},\\
\end{aligned}
\] 
we suppose that 
\be\label{ZM}
\text{$Z_{Q}\not=\emptyset$}
\ee
and 
\be\label{M1}
\VV^* ({Q_{z}^{*}})^{-1}\WW^*=\WW Q_{\bar z}^{-1}\VV \,, \qquad z\in Z_{Q}.
\ee
\begin{remark}\label{adj}
Notice that the left hand side of \eqref{M1} is well defined: since $z\in Z_{Q}$, $Q_{z}^{-1}$ is bounded and so its adjoint exists (and is bounded); moreover $\ker(Q_{z}^{*})=\ran(Q_{z})^{\perp}=\X^{\perp}=\{0\}$ and so $Q_{z}^{*}$ is invertible and $(Q_{z}^{*})^{-1}=(Q_{z}^{-1})^{*}$. 
\end{remark}
\begin{remark} Notice that $Q_{w}$, $w\in Z_{Q}$, is closed since it is the inverse of a bounded (hence closed) operator. Then $Q_{z}$, $z\in \rho(A_{0})$, is closed since, by \eqref{M2}, it differs from $Q_{w}$ by a bounded operator. 
\end{remark}
\begin{remark}\label{uno} Notice that if $\VV=\uno$ (or $\WW=\uno$)  then 
\eqref{M1} follows from $Q_{z}^{*}\WW=\WW^{*}Q_{\bar z}$ (or $\VV Q_{z}^{*}=Q_{\bar z}\VV^*$). 
\end{remark}
The set of maps satisfying \eqref{M0}-\eqref{M1} is not void, we give some examples. Below we 
consider a Weyl function $M:\rho(A_{0})\to\B(\K^{*},\K)$, $z\mapsto M_{z}$, i.e.,  a $\B(\K^{*},\K)$-valued map such that 
\be\label{W}
M_{z}^{*}=M_{\bar z}\,,\qquad M_{z}-M_{w}=(z-w)G^{*}_{\bar w}G_{z}\,.
\ee 
The canonical representation is $M_{z}:=\tau((G_{z_{0}}+G_{\bar z_{0}})/2-G_{z})$, $z_{0}\in\rho(A_{0})$, (see \cite[Lemma 2.2]{P01}; it is well defined thanks to \eqref{reg}). In the case $\tau$ has a bounded extension to $\ran(G_{z})$ (eventually considering a range space for $\tau$ larger than the original $\K$), one can take $M_{z}:=-\tau G_{z}$. 
\begin{example} Let $\X$ be a closed subspace of $\K$ and let $\pi:\K\to\K$, $\ran(\pi)=\X$, be the corresponding orthogonal projector. Then $\pi^{*}:\K^{*}\to\K^{*}$ is an orthogonal projector as well. Let us set $\Y:=\X^{*}=\ran(\pi^{*})$, $\VV:=\pi:\K\to\X$, $\WW:=\pi^{*}:\Y\to\K^{*}$. Given $\Theta:\dom(\Theta)\subseteq\X^{*}\to\X$ self-adjoint and a Weyl function $M:\rho(A_{0})\to\B(\K^{*},\K)$, $z\mapsto M_{z}$, we define $Q_{z}:\dom(\Theta)\subseteq\Y\to\X$ by $Q_{z}:=\Theta+\VV M_{z}\WW$. 
 If one further supposes that $\tau$ is surjective, i.e., $\ran(\tau)={\K}$, then $\CO\backslash\RE\subseteq Z_{Q}$ (see \cite[Proposition 2.1]{P01}, \cite[Theorem 2.1]{P08}). $Q:z\mapsto Q_{z}$ satisfies \eqref{M0}, \eqref{M2} and $Q_{z}^{*}=Q_{\bar z}$ by \eqref{W}. So 
$(Q^{-1}_{z})^{*}=(Q^{*}_{z})^{-1}=Q^{-1}_{\bar z}$, $z\in Z_{Q}$. Since $\VV$ and $\WW$ are orthogonal projectors, this gives \eqref{M1}. For explicit examples where such kind of maps appear in applications to partial differential operators, see \cite{Grubb}, \cite{P08}, \cite{PR}, \cite{MMM}, \cite{GM}, \cite{Grubb12}, \cite{MSP-JDE1}, \cite{CPP}, \cite{MPS} and references therein. As Theorem \ref{TeoInv} below shows, it is not necessary to suppose $\ran(\tau)=\K$ whenever one knows that $Z_{Q}\not=\emptyset$. 
\end{example}  
\begin{example} Let $\alpha\in\B(\K,\K^{*})$, $\alpha^{*}=\alpha$, and let $M:\rho(A_{0})\to\B(\K^{*},\K)$, be a Weyl function. Suppose that there exists $c>0$ such that $\| M_{z}\|_{\B(\K^{*},\K)}<\|\alpha\|_{\B(\K,\K^{*})}^{-1}$ whenever $|\text{Im}(z)|>c$. Then define $Q_{z}\in \B(\K^{*})$ by $Q_{z}:=-(\uno-\alpha M_{z})$. It is immediate to check (also use Remark \ref{uno}) that $Q:z\mapsto Q_{z}$ satisfies \eqref{M0}-\eqref{M1} with $\X=\Y=\K^{*}$, $\VV=\alpha$, $\WW=\uno$ and $Z_{Q}=\{z\in\rho(A_{0}):|\text{Im}(z)|>c\}$. Such kind of maps appears in the definition of Laplacians with $\delta$-type potentials supported on a compact hypersurface (see \cite{BLL}, \cite[Section 5.4]{MP}, \cite{MPS-JDE2} and references therein); in such references  it is proven that $\CO\backslash\RE\subseteq Z_{Q}$ by analytic Fredholm theory ($M_{z}$ is a compact operators in these examples). As Theorem \ref{TeoInv} below shows, this  is not necessary, $Z_{Q}\not=\emptyset$ suffices. In the not compact case, for Laplacians with $\delta$-type potentials supported on a deformed plane, in \cite[Lemma 3.6]{CFP} it is proven $\CO\backslash\RE\subseteq Z_{Q}$ whenever the deformation is in  $C_{0}^{1,1}(\RE^{2})$, while $Z_{Q}\not=0$ whenever the deformation is in $C_{0}^{0,1}(\RE^{2})$, i.e., is Lipschitz (see \cite[Lemma 3.5]{CFP}). By Theorem \ref{TeoInv}, the latter hypothesis suffices to prove that $Z_{Q}=\rho(A_{0})\cap \rho(A_{\Lambda})$. 
\end{example}
\begin{example} Let $\VV\in\B(\K,\K^{*})$, $\WW\in\B(\K^{*},\K)$ such that $\VV^{*}\WW^{*}=\WW\VV$ and let $M:\rho(A_{0})\to\B(\K^{*},\K)$, $z\mapsto M_{z}$, be a Weyl function. Suppose that there exists $c>0$ such that $\| M_{z}\|_{\B(\K^{*},\K)}<\|\VV\|_{\B(\K,\K^{*})}^{-1}\|\WW\|_{\B(\K^{*},\K)}^{-1}$ whenever $|\text{Im}(z)|>c$. Then define $Q_{z}\in \B(\K^{*})$ by $Q_{z}:=-(\uno-\VV M_{z}\WW)$. It is immediate to check that $Q:z\mapsto Q_{z}$ satisfies \eqref{M0}, \eqref{M2} and  $Z_{Q}=\{z\in\rho(A_{0}):|\text{Im}(z)|>c\}$ with $\X=\Y=\K^{*}$. As regards \eqref{M1}, it holds by
\begin{align*}
&\VV^{*}(Q_{z}^{*})^{-1}\WW^{*}=-\VV^{*}(\uno-\WW^{*}M_{\bar z}\VV^{*})^{-1}\WW^{*}=
-\VV^{*}\left(\sum_{n=0}^{\infty}(\WW^{*}M_{\bar z}\VV^{*})^{n}\right)\WW^{*}\\
=&-\sum_{n=0}^{\infty}\VV^{*}\underbrace{\WW^{*}M_{\bar z}\VV^{*}\dots \WW^{*}M_{\bar z}\VV^{*}}_{\text{$n$-times}}\WW^{*}=-\sum_{n=0}^{\infty}\WW\underbrace{\VV M_{\bar z}\WW\dots \VV M_{\bar z}\WW}_{\text{$n$-times}}\VV\\
=&-\WW(\uno-\VV M_{\bar z}\WW)^{-1}\VV=\WW Q_{\bar z}^{-1}\VV\,.
\end{align*}
Alike maps appear in \cite[Appendix B]{Albeverio} and produce resolvent formulae similar to the (Kato-)Konno-Kuroda one (see \cite{Kato}, \cite{KK}). However in \cite[Appendix B]{Albeverio} it is assumed that the map $E^{*}F$, where $F:=V\tau$, $E:=W^{*}\tau$,  is infinitesimally bounded with respect to $|A_{0}|^{1/2}$ and that $M_{z}$ is compact.  As Theorem \ref{TeoInv} below shows, these hypotheses  are not necessary, $Z_{Q}\not=\emptyset$ suffices.
\end{example}
\begin{example} Let $Q:\rho(A_{0})\to\C(\Y,\X)$ be any map satisfying \eqref{M0}-\eqref{M1} with $\VV=\uno$ (or $\WW=\uno$) and let $B\in\B(\Y,\X)$ such that $B^{*}\WW=\WW^{*}B$ (or $\VV B^{*}=B\VV^{*}$). Define $\widetilde Q_{z}:=B+Q_{z}$. For any $z\in Z_{Q}$ one has $\widetilde Q_{z}=(1+BQ_{z}^{-1})Q_{z}$. Suppose that 
\begin{align*}\widetilde Z_{Q}&&\\
:=&\{z\in Z_{Q}:\text{$1+BQ_{z}^{-1}$ and $1+BQ_{\bar z}^{-1}$ are continuous bijections from $\X$ onto $\X$}\}
\end{align*} 
is not void. Then $\widetilde Q:z\mapsto \widetilde Q_{z}$ satisfies \eqref{M0}-\eqref{M1}. A map of such kind is used in \cite[section 5.5]{MP} to describe  Laplacians with $\delta'$-type potentials supported on compact Lipschitz hypersurfaces. There $Q_{z}^{-1}$ is  compact and it is proven that $\CO\backslash\RE\subseteq\widetilde Z_{Q}$ by analytic Fredholm theory. As Theorem \ref{TeoInv} below shows, $\widetilde Z_{Q}\not=\emptyset$ suffices to prove that $Z_{\widetilde Q}=\rho(A_{0})\cap \rho(A_{\Lambda})$.
\end{example}
Given $Q$ which satisfies \eqref{M0}-\eqref{M1}, it is immediate to check (also use Remark \ref{adj}) that 
$$\Lambda^{Q}:Z_{Q}\to\B({\X},{\Y})\,,\qquad
\Lambda^{Q}_{z}:=Q_{z}^{-1}\,,
$$ 
satisfies \eqref{Lambda1}  and \eqref{Lambda2} and thus we can apply Theorem \ref{teo}. 
From now on we use the notations 
$$A_{Q}:=A_{\Lambda^{Q}}\,,\qquad  R^{Q}_{z}:=(-A_{Q}+z)^{-1}\,,\quad z\in \rho(A_{Q})\,.
$$ 
According to \eqref{Lambda4}, we can introduce the analytic extension of $\Lambda^{Q}$ given by 
$$
\widehat\Lambda^{Q}:\rho(A_{Q})\to\B({\X},{\Y})\,,
$$
\be\label{LambdaM}
\widehat\Lambda^{Q}_{z}:=Q_{w}^{-1}+(w-z)Q_{w}^{-1}\VV G_{\bar w}^{*}\left(1+(w-z) R^{Q}_{z}\right)G_{w}\WW Q_{w}^{-1}\,,\quad w\in Z_{Q}\,.
\ee
\begin{remark} Notice that, since we are not supposing that $Z_{Q}$ contains an accumulation point, the extension  $\widehat\Lambda^{Q}$ could depend on the choice of the point $w\in Z_\Lambda$. This is not the case, as  Theorem \ref{TeoInv} shows. 
\end{remark}
At first we provide the following 
\begin{lemma} Let $\Lambda:Z_{\Lambda}\to \B({\X} ,{\Y})$ be as in Theorem \ref{teo}. Then,  for any $w\in Z_{\Lambda}$ and for any $z\in\rho(A_{0})\cap\rho(A_{\Lambda})$, one has 
\be\label{inv}
R^{\Lambda}_{z}-R^{0}_{z}
=\left(1+(w-z)R^{\Lambda}_{z}\right)G_{w}\WW \Lambda_{w}\VV G^{*}_{\bar w}
\left(1+(w-z)R_{z}^{0}\right).
\ee
\end{lemma}
\begin{proof} In the case $z=w$, \eqref{inv} reduces to \eqref{resolvent}. Hence it suffices to prove the thesis in the case $z\not=w$. By functional calculus, it is immediate to check that
\be\label{fc}
(w-z)\left(1+(w-z)R_{z}\right)=\left(-R_{w}+\frac1{w-z}\right)^{-1}
\ee
for any  $w,z\in \rho(A)$, $w\not= z$, where $R_{z}:=(-A+z)^{-1}$  is the resolvent of a self-adjoint operator $A$. Thus, by \eqref{fc} and \eqref{resolvent},
\begin{align*}
&(w-z)^{2}(R^{\Lambda}_{z}-R^{0}_{z})\\
=&(w-z)\left(1+(w-z)R^{\Lambda}_{z}\right)-(w-z)\left(1+(w-z)R^{0}_{z}\right)\\
=&\left(-R^{\Lambda}_{w}+\frac1{w-z}\right)^{-1}-\left(-R^{0}_{w}+\frac1{w-z}\right)^{-1}\\
=&\left(-R^{\Lambda}_{w}+\frac1{w-z}\right)^{-1}(R^{\Lambda}_{w}-R^{0}_{w})\left(-R^{0}_{w}+\frac1{w-z}\right)^{-1}\\
=&\left(-R^{\Lambda}_{w}+\frac1{w-z}\right)^{-1}G_{w}\WW  \Lambda_{w} \VV G^{*}_{\bar w}\left(-R^{0}_{w}+\frac1{w-z}\right)^{-1}\\
=&(w-z)^{2}\left(1+(w-z)R^{\Lambda}_{z}\right)G_{w}\WW \Lambda_{w}\VV G^{*}_{\bar w}
\left(1+(w-z)R_{z}^{0}\right).
\end{align*}
\end{proof}
\begin{remark} Notice that by the exchange $R_{z}^{\Lambda}\leftrightarrow R^{0}_{z}$ in the above proof one gets the alternative identity
 \be\label{inv'}
R^{\Lambda}_{z}-R^{0}_{z}
=\left(1+(w-z)R^{0}_{z}\right)G_{w}\WW \Lambda_{w}\VV G^{*}_{\bar w}
\left(1+(w-z)R_{z}^{\Lambda}\right).
\ee
\end{remark}
The previous Lemma provides an essential ingredient in the proof of our  main result:
\begin{theorem}\label{TeoInv} Let  $Z_{Q}\not=\emptyset$, $Q:\rho(A_{0})\to \C({\Y} ,{\X})$ a map statisfying \eqref{M0}, \eqref{M2}, and \eqref{M1}.  Then $Z_{Q}=\rho(A_{0})\cap\rho(A_{Q})$ and  for any $z\in \rho(A_{0})\cap\rho(A_{Q})$ one has $Q_{z}^{-1}=\widehat\Lambda^{Q}_{z}$. Moreover  the resolvent formula
$$
(-A_{Q}+z)^{-1}=R^{0}_{z}+G_{z}\WW Q_{z}^{-1}\VV G^{*}_{\bar z}\,,\qquad  z\in \rho(A_{0})\cap\rho(A_{Q})
\,,$$
holds true.
\end{theorem}
\begin{proof} 
The first statement of the theorem  is equivalent to show that the two identities 
$\widehat\Lambda_{z}^{Q}Q_{z}=1_{\Y}$ and $Q_{z}\widehat\Lambda_{z}^{Q}=1_{{\X}}$ hold true for any $z\in \rho(A_{0})\cap\rho(A_{Q})$, $z\not=w\in Z_{Q}$. 

\par By \eqref{LambdaM} and \eqref{M2}, one gets
\begin{align*}
&\widehat\Lambda_{z}^{Q}Q_{z}\\
=&\left(Q_{w}^{-1}+(w-z)Q_{w}^{-1}\VV G_{\bar w}^{*}\left(1+(w-z) R^{Q}_{z}\right)G_{w}\WW Q_{w}^{-1}\right)(Q_{w}+(Q_{z}-Q_{w}))\\
=&1+(w-z)Q_{w}^{-1} \VV G_{\bar w}^{*}\left(1+(w-z) R^{Q}_{z}\right)G_{w} \WW
-(w-z)Q_{w}^{-1} \VV G^{*}_{\bar w}G_{z} \WW \\
&-
(w-z)^{2}Q_{w}^{-1}\VV G_{\bar w}^{*}\left(1+(w-z) R^{Q}_{z}\right)G_{w}\WW Q_{w}^{-1} \VV G^{*}_{\bar w}G_{z} \WW \,.
\end{align*}
Hence, by \eqref{Gzw'} and \eqref{inv}, 
\begin{align*}
&(w-z)^{-2}\big(\widehat\Lambda_{z}^{Q}Q_{z}-1\big)\\
=&(w-z)^{-1}Q_{w}^{-1} \VV G_{\bar w}^{*}\left(\left(1+(w-z) R^{Q}_{z}\right)-\left(1+(w-z)R^0_{z}\right)\right)G_{w} \WW \\
&-Q_{w}^{-1}\VV G_{\bar w}^{*}\left(1+(w-z) R^{Q}_{z}\right)G_{w}\WW Q_{w}^{-1}\VV G^{*}_{\bar w}\left(1+(w-z)R^0_{z}\right)G_{w}\WW \\
=&Q_{w}^{-1}\VV G_{\bar w}^{*}\left((R^{Q}_{z}-R^0_{z})-
\left(1+(w-z) R^{Q}_{z}\right)G_{w}\WW Q_{w}^{-1}\VV G^{*}_{\bar w} \left(1+(w-z)R^0_{z}\right)\right)G_{w}\WW 
\\
=&0\,.
\end{align*}
The proof of the other identity is almost the same. At first let us notice that $Q_{z}\widehat\Lambda^{Q}_{z}$ is well defined  since, by definition \eqref{LambdaM} and \eqref{M2}, $$\ran({\widehat\Lambda^{Q}_{z}})\subseteq\ran(Q_{w}^{-1})=\dom(Q_{w})={\mathsf D}=\dom(Q_{z})\,.$$
By \eqref{LambdaM} and \eqref{M2}, one gets
\begin{align*}
&Q_{z}\widehat\Lambda_{z}^{Q}\\
=&(Q_{w}+(Q_{z}-Q_{w}))\left(Q_{w}^{-1}+(w-z)Q_{w}^{-1}\VV G_{\bar w}^{*}\left(1+(w-z) R^{Q}_{z}\right)G_{w}\WW Q_{w}^{-1}\right)\\
=&1+(w-z) \VV G_{\bar w}^{*}\left(1+(w-z) R^{Q}_{z}\right)G_{w}\WW Q_{w}^{-1}
-(w-z)\VV G^{*}_{\bar w}G_{z} \WW Q_{w}^{-1}\\
&-
(w-z)^{2}\VV G_{\bar w}^{*}G_{z}\WW Q_{w}^{-1}\VV G^{*}_{\bar w}\left(1+(w-z) R^{Q}_{z}\right)G_{w}\WW Q_{w}^{-1}\,.
\end{align*}
Hence, by \eqref{Gzw'} and \eqref{inv'}, 
\begin{align*}
&(w-z)^{-2}\big(Q_{z}\widehat\Lambda_{z}^{Q}-1)\\
=&(w-z)^{-1}\VV G_{\bar w}^{*}\left(\left(1+(w-z) R^{Q}_{z}\right)-\left(1+(w-z)R^0_{z}\right)\right)G_{w}\WW Q_{w}^{-1}\\
&-\VV G_{\bar w}^{*}\left(1+(w-z) R^{0}_{z}\right)G_{w}\WW Q_{w}^{-1}\VV G^{*}_{\bar w}\left(1+(w-z)R^{Q}_{z}\right)G_{w}\WW Q_{w}^{-1}\\
=& \VV G_{\bar w}^{*}\left((R^{Q}_{z}-R^{0}_{z})- \left(1+(w-z) R^{0}_{z}\right)G_{w}\WW Q_{w}^{-1}\VV G^{*}_{\bar w}\left(1+(w-z)R^{Q}_{z}\right)\right)G_{w}\WW Q_{w}^{-1}
\\
=&0\,.
\end{align*}
\par To conclude the proof of the theorem we must show that $\widehat\Lambda^{Q}_{z}$ satisfies the identities \eqref{Lambda1} and \eqref{Lambda2} for all  $z,w\in \rho(A_{0})\cap\rho(A_{Q})$. These are immediate consequences of Remark \ref{Lambda*} ($\widehat\Lambda_{z}^{Q}=Q_{z}^{-1}$ does not depend on $w$) and \eqref{M2}. 
\end{proof}
\end{section}
\begin{remark}\label{nodense} Notice that in the proof of the previous theorem we did not use  \eqref{dense}. This hypothesis is only needed in the proof of Theorem \ref{teo}. In case \eqref{krein} still holds, then the statements in Theorem \ref{TeoInv} retain their validity without requiring $\overline{\ker(\tau)}=\H$. 
\end{remark}

\end{document}